\documentclass[journal]{IEEEtran}

\usepackage{pmat}

\usepackage{float}

\usepackage{algorithmic}
\usepackage{algorithm}
\usepackage[table]{xcolor}
\usepackage{multirow}
\usepackage{amssymb,amsmath,amsthm,bm,amsfonts}
\usepackage{cases}
\usepackage{color}
\usepackage{lineno,hyperref}
\hypersetup{hidelinks} 
\usepackage{cite}               
\usepackage{graphicx}
\usepackage{subfigure}
\usepackage{doi}

\newtheorem{theorem}{Theorem}
\newtheorem{lemma}{Lemma}

\newtheorem{assumption}{Assumption}

\newtheorem{example}{Example}

\usepackage{stfloats}
\fnbelowfloat

\begin{document}
%
\title{Distributed Optimal Control and Application to Consensus of Multi-Agent Systems}

%
%
\author{Liping Zhang, 
        Juanjuan Xu, 
       Huanshui Zhang
      ~and Lihua Xie
\thanks{This work was supported by the Original Exploratory Program Project of National Natural Science Foundation of China (62250056), National Natural Science Foundation of China (62103240),
Youth Foundation of Natural Science Foundation of Shandong Province (ZR2021QF147),
Foundation for Innovative Research Groups of the National Natural Science Foundation of China(61821004),
Major Basic Research of Natural Science Foundation of Shandong Province (ZR2021ZD14),
 High-level Talent Team Project of Qingdao West Coast New Area (RCTD-JC-2019-05), Key Research and Development Program of Shandong Province (2020CXGC01208), and
Science and Technology Project of Qingdao West Coast New Area (2019-32, 2020-20, 2020-1-4).}
\thanks{L. Zhang and H. Zhang are with the College of Electrical Engineering and Automation,
Shandong University of Science and Technology, Qingdao 266590, China (e-mail:lpzhang1020@sdust.edu.cn; hszhang@sdu.edu.cn).}
\thanks{J. Xu is with the School of Control Science and Engineering, Shandong University, Jinan 250061, China (e-mail:juanjuanxu@sdu.edu.cn).}
\thanks{L. Xie is with the School of Electrical and Electronic Engineering, Nanyang Technological University, Singapore 639798, Singapore (e-mail:ELHXIE@ntu.edu.sg).}
}

\maketitle

\begin{abstract}
This paper develops a novel approach to the consensus problem of  multi-agent systems
by minimizing a weighted state error with neighbor agents via linear quadratic (LQ) optimal control theory.
Existing consensus control algorithms only utilize the current state of each agent,
and the design of distributed controller depends on nonzero eigenvalues of  the communication topology.
The presented optimal consensus controller is obtained by solving Riccati equations
and designing appropriate observers to account for agents' historical state information.
It is shown that the corresponding cost function under the proposed controllers is asymptotically optimal.
Simulation examples demonstrate the effectiveness of the proposed scheme, and a much faster convergence speed than the conventional consensus methods.
  Moreover,  the new method avoids computing nonzero eigenvalues of the communication topology as in the traditional consensus methods.
\end{abstract}

\begin{IEEEkeywords}
 Consensus,
 Distributed control,
 Observer,
 Heterogeneous multi-agent system
\end{IEEEkeywords}

%
\IEEEpeerreviewmaketitle

\section{Introduction}
Cooperative control problem for multi-agent systems has  attracted  considerable attentions from different scientific communities in recent years.
Multiple agents can coordinate with each other via communication topology to accomplish tasks that may be difficult for single agent,
and its potential applications include unmanned aerial vehicles, satellite formation,
 distributed robotics and wireless sensor networks \cite{Ren2007a,Olfati-Saber2007,Yang2022}.
 In the area of cooperative control of multi-agent systems, consensus is a fundamental and crucial problem,
 which refers to designing an appropriate distributed control protocol to steer all agents to achieve an agreement on certain variable\cite{Li2014}.
Thus, the consensus problem has been widely studied by numerous researchers from various perspectives.

Homogeneous multi-agent systems means that all agents have identical system dynamics, which mainly includes leaderless consensus and leader-follower consensus.
\cite{Olfati-Saber2004} proposed a general framework of the consensus problem for networks of first-order integrator agents with switching topologies
based on the relative states.
\cite{Ren2005} derived a sufficient condition, which was more relaxed than that in \cite{Olfati-Saber2004}, for achieving consensus of multi-agent systems.
Extending first-order consensus protocols in \cite{Ren2005},
the author in \cite{Ren2007} further studied distributed leader-following consensus algorithms for second-order integrators.
\cite{You2011} considered the consensusability of discrete-time multi-agent systems, and  an upper bound of the eigenratio (the ratio of the second smallest to the
largest eigenvalues of the graph Laplacian matrix) was derived to characterize the convergence rate.
\cite{Li2014} proposed a consensus region approach  to designing distributed adaptive consensus protocols with undirected and directed graph for general continuous-linear dynamics.
\cite{Feng2022} recently derived an optimal consensus region over directed communication graphs with a diagonalizable Laplacian matrix.
Besides, variants of these algorithms are also currently applied to tackle communication uncertainties, such as fading communication channels \cite{Xu2021},
packet loss \cite{Xu2019} and communication time-delays \cite{Liu2011}.
It should be pointed out that the  aforementioned consensus control protocols merely use each agent and its neighbor's current state information,
and ignore their historical state information.
Additionally, the solvability condition of the consensus gain matrix $K$ is dependent on the nonzero eigenvalues of Laplacian matrix or even
requires the communication topology to be a complete graph \cite{Cao2010}.
In particular, when agent number is large, the eigenvalues of the corresponding Laplacian matrix are difficult to be determined,
even if the eigenvalues are computable, since their calculation still imposes a significant computational burden.
\cite{Chen2020} presented a distributed consensus algorithm based on optimal control theory,
 while the state weight matrix in given performance index is a special form.

On the other hand, many actual systems are heterogeneous where system dynamics are different.
So far, the distributed feedforward control \cite{Su2012,Huang2017,Lu2017}
and  internal model  principle \cite{Wieland2011} are commonly used to solve the cooperative output regulation problem.
These tools are also generalized for dealing with robust output regulation, switching networks and cooperative-compete networks\cite{Li2015,Bi2022,Meng2017,Liu2018,Yaghmaie2017}.
In fact, the essence of both the algorithms can be attributed to two aspects:
first, the reference generator \cite{Scardovi2009} or the distributed observer estimating the reference system's state is a critical technology for designing distributed controllers;
second, the solvability conditions of output regulator equations or transmission zero conditions of the system are also necessary for solving the output consensus problem.

Motivated by the above analyses, in this paper,
 we study the consensus problem of discrete-time linear heterogeneous multi-agent systems with a novel consensus control protocol based on
 LQ optimal control theory.
 Compared with the existing results, the main contributions of this work are:
 1) We develop a novel consensus algorithm by minimizing the  weighted state errors of different neighbor agents.
 An optimal consensus controller with the observer incorporating each agent's historical state information is designed by solving  Riccati equations.
The corresponding global cost function under the proposed controllers is shown to be asymptotically optimal.
 2) The proposed new consensus controller can achieve much faster consensus speed than the traditional consensus method,
 and avoid computing nonzero eigenvalues of the Laplacian matrix associated with the communication topology.


The following notations will be used throughout this paper:
$\mathbb{R}^{n\times m}$ represents the set of $n\times m$-dimensional real matrices.
$I$ is the identity matrix of a given dimension.
$\mbox{diag}\{a_1,a_2,\cdots,a_{N}\}$ denotes the diagonal matrix with diagonal elements being $a_1,\cdots,a_{N}$.
$\rho(A)$ is the spectral radius of matrix $A$.
$\otimes$ denotes the Kronecker product.

Let the interaction among $N$ agents be described by a directed graph
$\mathcal{G}=\{\mathcal{V},\mathcal{E},\mathcal{A}\}$,
where $\mathcal{V}=\{1,2,\cdots,N\}$ is the set of vertices (nodes),
$\mathcal{E}\subseteq \mathcal{V}\times \mathcal{V}$ is the set of edges,
and $\mathcal{A}=[a_{ij}]\in  \mathbb{R}^{N\times N}$ is the signed weight matrix of $\mathcal{G}$,
$a_{ij}\neq 0$ if and only if the edge $(v_{j},v_{i})\in \mathcal{E}$,
and we assume that the graph has no self-loop, i.e., $a_{ii}=0$.
The neighbor of $v_{i}$ is denoted by $\mathcal{N}_{i}=\{j|(v_{j},v_{i})\in \mathcal{E}\}$.
The Laplacian matrix $\mathcal{L}= [l_{ij}]_{N\times N}$ associated with the adjacency matrix $\mathcal{A}$ is defined as
$l_{ii}=\sum\limits_{j\in N_{i}} a_{ij}$, $l_{ij}=-a_{ij}$ for $i\neq j$.
A directed path form $v_{i}$ to $v_{j}$ is represented by a sequence of edges $(v_{i},v_{i1}),(v_{i1},v_{i2}),\cdots,(v_{im},v_{j})$.
A directed graph is strongly connected if there exists a directed path between any pair of distinct nodes.

\section{Preliminary and Problem Formulation} \label{sec:preliminary}

\subsection{Problem Formulation}
 We consider a  heterogeneous  discrete-time multi-agent system consisting of $N$ agents over a directed graph $\mathcal{G}$ with the dynamics of each agent given by
\begin{align}\label{homogeneous multi-agent system}
x_{i}(k+1)=Ax_{i}(k)+B_{i}u_{i}(k), i=1,2,\cdots,N
\end{align}
where $x_{i}(k)\in \mathbb{R}^{n}$ and $u_{i}(k)\in \mathbb{R}^{m_{i}}$ are the state and the input of each agent.
$A\in \mathbb{R}^{n\times n}$ and $B_{i}\in \mathbb{R}^{n\times m_{i}}$ are the coefficient matrices.

The cost function of multi-agent systems \eqref{homogeneous multi-agent system} is given by
\begin{align}\label{cost-function}
J(s,\infty)&=\sum_{k=s}^{\infty}\left(\sum_{i=1}^{N} \sum_{j\in N_{i}}(x_{i}(k)-x_{j}(k))^{T}Q(x_{i}(k)-x_{j}(k))
\right.
\notag\\
&\quad \left.+ \sum_{i=1}^{N} u_{i}^{T}(k)R_{i}u_{i}(k)\right),
\end{align}
where $Q\geq 0$ and $R_{i}>0$ are  weighting matrices.

We aim to design a distributed control protocol $u_{i}(k)$
based on the available information from neighbors in \eqref{homogeneous multi-agent system} to minimize the performance \eqref{cost-function}.

Based on the optimal control theory, it is  clear that if the optimal controller exists,
it must have that
\begin{align}
\lim_{k\to \infty} \|x_{i}(k)-x_{j}(k)\|=0, i=1,\cdots, N.
\end{align}
In other word,
multi-agent systems \eqref{homogeneous multi-agent system} achieve consensus,
and the protocol is termed as an optimal control based protocol, which is completely different from classical approaches.
In fact, the commonly used consensus protocol \cite{Olfati-Saber2004,Ren2005} for multi-agent systems is designed as:
\begin{align}\label{general-controller}
u_{i}(k)=F\sum_{j\in N_{i}}a_{ij}(x_{j}(k)-x_{i}(k)),
\end{align}
where $F$ is a feedback gain matrix, which actually is dependent on $\lambda_{2}(\mathcal{L})$ and $\lambda_{N}(\mathcal{L})$ \cite{You2011}.
That is to say, one needs to solve non-zero eigenvalues for the Laplacian matrix $\mathcal{L}$ associated with the communication topology to determine the feedback gain $F$.

Different from the commonly used consensus protocol where the protocol is artificially defined,
the protocol in this paper is derived by optimizing a given LQ performance,
and the performance index \eqref{cost-function} is more general with a positive semi-definite weight matrix $Q\geq 0$.

\subsection{Preliminary}
Define the neighbor error variable among agents as: $e_{ij}(k)=x_{i}(k)-x_{j}(k)$.
Then, it can be obtained from \eqref{homogeneous multi-agent system}  that
\begin{align}\label{error-dynamics}
e_{ij}(k+1)=Ae_{ij}(k)+B_{i}u_{i}(k)-B_{j}u_{j}(k).
\end{align}
Let $\delta_{i}(k)=\begin{bmatrix}e_{ij_1}^{T} & e_{ij_2}^{T} \cdots& e_{ij_\ell}^{T}
\end{bmatrix}^{T}$ be the error vector between the $i$-th agent and its neighbor agent $\mu$ with $\mu=j_{1},\cdots,j_{\ell}$.
By stacking the error vectors,
the global error dynamics for the multi-agent system \eqref{homogeneous multi-agent system} has the form
\begin{align}\label{gloabl-error-system}
e(k+1)&=\tilde{A}e(k)+\sum_{i=1}^{N}\bar{B}_{i}u_{i}(k),
 \\
\mathcal{Y}_{i}(k)&=H_{i}e(k),
\end{align}
where
$e(k)=\begin{bmatrix}
 \delta_{1}^{T}(k) &\delta_{2}^{T}(k) &\cdots & \delta_{N}^{T}(k)
\end{bmatrix}^{T}$ is the global error vector, $\mathcal{Y}_{i}(k)$ is measurement.
$\tilde{A}=I_{N}\otimes A$,
 $\bar{B}_{i}$ consists of $0_{n\times m_{i}},B_{i}$ and $-B_{i}$,
   and $H_{i}$  is composed of $0$ and $I_{n}$,  whose specific forms are dependent on the interaction among agents.
 Denote $\bar{B}=\begin{bmatrix}
\bar{B}_{1} & \bar{B}_{2} &\cdots & \bar{B}_{N}
\end{bmatrix}$,
and
$u(k)=\begin{bmatrix}
u_{1}^{T}(k),\cdots,u_{N}^{T}(k)
\end{bmatrix}^{T}.$

 The cost function \eqref{cost-function} is rewritten as
\begin{align}\label{cost-function-error}
J(s,\infty)&=\sum_{k=s}^{\infty}\left(\sum_{i=1}^{N} \sum_{j\in N_{i}}e_{ij}^{T}(k)Qe_{ij}(k)+
 \sum_{i=1}^{N} u_{i}^{T}(k)R_{i}u_{i}(k)\right)
\notag\\
&=\sum_{k=s}^{\infty}[e^{T}(k)\widetilde{Q}e(k)+u^{T}(k)Ru(k)]
\end{align}
with $\tilde{Q}=diag\{Q,Q,\cdots,Q\}\geq 0$ and  $R=diag\{R_{1},R_{2},\cdots,R_{N}\}>0$.


\begin{assumption}\label{graph-assumption}
The directed graph $\mathcal{G}$ is strongly connected.
\end{assumption}

In the ideal (complete graph) case, the error information $e(k)$ is available for all agents,
 the solvability of the optimal control problem for system \eqref{gloabl-error-system} with the cost function \eqref{cost-function-error} is
 equivalent to the following standard LQ optimal control problem \cite{Anderson1971}.
 Moreover, under the centralized optimal controller \eqref{centralized-optimal-control-error},
multi-agent system \eqref{homogeneous multi-agent system} is able to achieve consensus.

\begin{lemma}\cite{Anderson1971}\label{LQR-control-error}
 Suppose that error information $e(k)$ is available for all agents,
the optimal controller with respect to the cost function \eqref{cost-function-error} is given by
\begin{align}\label{centralized-optimal-control-error}
u^{*}(k)=K_{e}e(k),
\end{align}
where the feedback gain $K_{e}$ is given by
\begin{align}\label{feedback-gain-matrix-Ke}
K_{e}=-(R+\bar{B}^{T}P_{e}\bar{B})^{-1}\bar{B}^{T}P_{e}\tilde{A}
\end{align}
and $P_{e}$ is the solution of the following  ARE
\begin{align}\label{algebra-riccati-equation-error}
P_{e}=\tilde{A}^{T}P_{e}\tilde{A}+\tilde{Q}-\tilde{A}^{T}P_{e}\bar{B}(R+ \bar{B}^{T}P_{e}\bar{B})^{-1}\bar{B}^{T}P_{e}\tilde{A}.
\end{align}
The corresponding optimal cost function is
\begin{align}\label{cost-function-centralized-control-error-feedback}
J^{*}(s,\infty)=e^{T}(s)P_{e}e(s).
\end{align}
Moreover, if $P_{e}$ is the unique positive definite solution to \eqref{algebra-riccati-equation-error},
then $\tilde{A}+\bar{B}K_{e}$ is stable.
\end{lemma}

\section{Main Results}\label{sec:main-results}
\subsection{Consensus of multi-agent systems \eqref{homogeneous multi-agent system} based on relative error feedback}
In this subsection, we design a novel distributed observer-based controller  only using the relative error information from neighbors.

We design the following distributed controllers as
\begin{align} \label{error-controller}
u_{i}^{*}(k)&=K_{ei}\hat{e}_{i}(k), \quad i=1,2,\cdots,N,
 \end{align}
where  $\hat{e}_{i}(k),i=1,2,\cdots,N$ are distributed observers to estimate the global error $e(k)$ in system \eqref{gloabl-error-system},
which is based on the available information of agent $i$. Thus,
\begin{subequations}\label{observer-2}
\begin{align}\label{observer-2-1}
\hat{e}_{1}(k+1)&=\tilde{A}\hat{e}_{1}(k)+\bar{B}_{1}u_{1}^{*}(k)  +\bar{B}_{2}K_{e2}\hat{e}_{1}(k)
\notag\\
&\quad +\cdots+\bar{B}_{N}K_{eN}\hat{e}_{1}(k)
\notag\\
&\quad +\Upsilon_{1}(\mathcal{Y}_{1}(k)-H_{1}\hat{e}_{1}(k))
\\
& \quad \cdots \quad \cdots \quad \cdots
 \notag\\
\hat{e}_{i}(k+1)&=\tilde{A}\hat{e}_{i}(k)+\bar{B}_{1}K_{e1}\hat{e}_{i}(k) +\cdots
\notag\\
&\quad+ \bar{B}_{i-1}K_{e_{i-1}}\hat{e}_{i}(k) +\bar{B}_{i} u_{i}^{*}(k)
\notag\\
&\quad +\bar{B}_{i+1}K_{e_{i+1}}\hat{e}_{i}(k) +\cdots+\bar{B}_{N}K_{eN} \hat{e}_{i}(k)
\notag\\
&\quad +\Upsilon_{i}(\mathcal{Y}_{i}(k)-H_{i}\hat{e}_{i}(k)),
 \\
& \quad \cdots \quad \cdots \quad \cdots
 \notag \\ \label{observer-2-2}
\hat{e}_{N}(k+1)&=\tilde{A}\hat{e}_{N}(k)+\bar{B}_{1}K_{e1}\hat{e}_{N}(k) +\cdots
\notag\\
&\quad+ \bar{B}_{N-1}K_{e_{N-1}}\hat{e}_{N}(k) +\bar{B}_{N} u_{N}^{*}(k)
\notag\\
&\quad +\Upsilon_{N}(\mathcal{Y}_{N}(k)-H_{N}\hat{e}_{N}(k)),
\end{align}
\end{subequations}
with  $K_{ei}=\begin{bmatrix}
0 &\cdots & I& 0\cdots & 0
\end{bmatrix}K_{e}$, which
is obtained by solving an ARE \eqref{algebra-riccati-equation-error},
and the observer gain $\Upsilon_{i}$ is to be determined later to
ensure the stability of the observers.

\begin{theorem}\label{main-result-2}
  Consider the global error system \eqref{gloabl-error-system},
and the distributed control laws  \eqref{error-controller} and \eqref{observer-2}.
 If there exist observer gains $\Upsilon_{i},i=1,\cdots,N$ such that the matrix
\begin{align}\label{A_ec}
\tilde{A}_{ec}=\begin{bmatrix}
\Theta_1 & -\bar{B}_{2}K_{e2} &  \cdots & -\bar{B}_{N}K_{eN}\\
-\bar{B}_{1}K_{e1} & \Theta_2 &   \cdots & -\bar{B}_{N}K_{eN} \\
 \vdots &  \vdots & \ddots  & \vdots \\
-\bar{B}_{1}K_{e1} &  \cdots  & -\bar{B}_{N-1}K_{e_{N-1}} &  \Theta_N
\end{bmatrix}
\end{align}
is stable, where $\Theta_{i}=\tilde{A}+\bar{B}K_{e}-\bar{B}_{i}K_{ei}-\Upsilon_{i}H_{i}$.
Then the observers \eqref{observer-2} are stable under the controller \eqref{error-controller}, i.e.,
\begin{align}\label{observer-error-vector-error-form}
\lim_{k\to \infty}\|\hat{e}_{i}(k)-e(k)\|=0.
\end{align}
Moreover,  if the Riccati equation \eqref{algebra-riccati-equation-error} has a positive definite solution $P_{e}$,
 under the distributed feedback controllers \eqref{error-controller},
the multi-agent systems \eqref{homogeneous multi-agent system} can achieve consensus.
 \end{theorem}

\begin{proof}
Denoting observer error vectors
\begin{align}
\tilde{e}_{i}(k)= e(k)-\hat{e}_{i}(k).
\end{align}
Then, combining system \eqref{gloabl-error-system} with observers \eqref{observer-2}, one obtains
\begin{align}\label{gloabl-error-closed-loop-system}
e(k+1)&=(\tilde{A}+\bar{B}K_{e})e(k)- \bar{B}_{1}K_{e1}\tilde{e}_{1}(k)
\notag\\
&\quad -\bar{B}_{2}K_{e2}\tilde{e}_{2}(k)-\cdots
 -\bar{B}_{N}K_{eN}\tilde{e}_{N}(k)
\end{align}
and
\begin{subequations}\label{observer-error-systems-gloabl}
\begin{align}
 \tilde{e}_{1}(k+1)&=(\tilde{A}+\bar{B}K_{e}-\bar{B}_{1}K_{e1}-\Upsilon_{1}H_{1})\tilde{e}_{1}(k)
\notag\\
&\quad -\bar{B}_{2}K_{e2}\tilde{e}_{2}(k)-\cdots-\bar{B}_{N}K_{eN}\tilde{e}_{N}(k)
\\
&\quad \cdots \quad \cdots \quad \cdots
\notag\\
\tilde{e}_{i}(k+1)&=(\tilde{A}+\bar{B}K_{e}-\bar{B}_{i}K_{ei}-\Upsilon_{i}H_{i})\tilde{e}_{i}(k)
\notag\\
&\quad -\bar{B}_{1}K_{e1}\tilde{e}_{1}(k)-\cdots-\bar{B}_{i-1}K_{e_{i-1}}\tilde{e}_{i-1}(k)
\notag\\
&\quad
-\bar{B}_{i+1}K_{e_{i+1}}\tilde{e}_{i+1}(k)-\cdots-\bar{B}_{N}K_{eN}\tilde{e}_{N}(k)
\\
& \quad \cdots  \quad  \cdots \quad  \cdots
\notag\\
\tilde{e}_{N}(k+1)&=(\tilde{A}+\bar{B}K_{e}-\bar{B}_{N}K_{eN}-\Upsilon_{N}H_{N})\tilde{e}_{N}(k)
\notag\\
&\quad -\bar{B}_{1}K_{e1}\tilde{e}_{1}(k)-\cdots-\bar{B}_{N-1}K_{e_{N-1}}\tilde{e}_{N-1}(k).
\end{align}
\end{subequations}

According to \eqref{observer-error-systems-gloabl},
we have
\begin{align}
\tilde{e}(k+1)=\tilde{A}_{ec}\tilde{e}(k),
\end{align}
where
$\tilde{e}(k)=\begin{bmatrix}\tilde{e}_{1}^{T}(k), \tilde{e}_{2}^{T}(k),\cdots,\tilde{e}_{N}^{T}(k) \end{bmatrix}^{T}$.
Obviously, if there exist matrices $\Upsilon_{i}$ such that $\tilde{A}_{ec}$ is stable, then observer errors $\tilde{e}(k)$ converge to zero as $k\to \infty$,
i.e., Eq. \eqref{observer-error-vector-error-form} holds.
Furthermore,
it follows from \eqref{gloabl-error-closed-loop-system} and \eqref{observer-error-systems-gloabl} that
\begin{align}\label{closed-loop-error-system-1}
\begin{bmatrix}
e(k+1)\\
\tilde{e}(k+1)
\end{bmatrix}
=\bar{A}_{ec} \begin{bmatrix}
e(k)\\
\tilde{e}(k)
\end{bmatrix}
\end{align}
where $\bar{A}_{ec}=\begin{bmatrix}
\tilde{A}+\bar{B}K_{e} &  \Omega_{e} \\
 0 & \tilde{A}_{ec}
\end{bmatrix}$ and $\Omega_{e}=\begin{bmatrix}
 -\bar{B}_{1}K_{e1}& \cdots & -\bar{B}_{N}K_{eN}
\end{bmatrix}$.
Since $P_{e}$ is the positive definite solution to Riccati equation \eqref{algebra-riccati-equation-error},  then $\tilde{A}+\bar{B}K_{e}$ is stable,
 based on the LQ control theory, the consensus of multi-agent system \eqref{homogeneous multi-agent system} can be achieved.

The proof is completed.
\end{proof}

Observe from \eqref{closed-loop-error-system-1}  that since $K_e$  has been given in \eqref{feedback-gain-matrix-Ke},
  the consensus error dynamics is dependent on $\tilde{A}_{ec}$, which is determined by the observer gains $\Upsilon_{i}$.
   Thus, to speed up the convergence of consensus,
   the remaining problem is  to choose the  matrix of $\Upsilon_{i},i=1,2,\cdots,N$ such that the maximum eigenvalue of
$\tilde{A}_{ec}$  is as small as possible. To this end, let $\rho>0$
such that
\begin{align*}
\tilde{A}_{ec}^{T}\tilde{A}_{ec}\leq \rho I,
\end{align*}
or
\begin{align}\label{Lmi}
\begin{bmatrix}
  -\rho I   &  \tilde{A}_{ec}^{T} \\
  \tilde{A}_{ec} & -I
   \end{bmatrix} \leq 0,
\end{align}
where $\tilde{A}_{ec}$ is as in \eqref{A_ec}.
 Then the optimal gain matrices  $\Upsilon_{i}$ are chosen by:
 \begin{align}\label{choose-gamma}
  \min_{\Upsilon_{i}}\rho \quad  \quad \mbox{s.t.} \quad \eqref{Lmi},\quad i=1,\cdots,N.
 \end{align}

%

Next, we will derive the cost difference between the proposed distributed controller \eqref{error-controller}
and the centralized optimal control \eqref{centralized-optimal-control-error}, and then analyze the asymptotical optimal property of the corresponding cost function.

For the convenience of analysis, denote that
\begin{align*}
M_{e1}&=(\tilde{A}+\bar{B}K_{e})^{T}P_{e}\Omega_{e}-\Omega_{e1},
\notag\\
M_{e2}&=\begin{bmatrix}
K_{e1}^{T}R_{1}K_{e1} &  0 &\cdots & 0\\
0  & K_{e2}^{T}R_{2}K_{e2} &\cdots & 0\\
\vdots & \vdots& \ddots & \vdots\\
0 & 0 &\cdots & K_{eN}^{T}R_{N}K_{eN}
\end{bmatrix}
+\Omega_{e2},
\notag\\
\Omega_{e}&=
\begin{bmatrix}
-\bar{B}_{1}K_{e1} &  \cdots & -\bar{B}_{N}K_{eN}
\end{bmatrix},
\notag\\
\Omega_{e1}&=\begin{bmatrix}
K_{e1}^{T}R_{1}K_{e1} &  \cdots & K_{eN}^{T}R_{N}K_{eN}
\end{bmatrix},
\notag\\
\Omega_{e2}&=\Omega_{e}^{T}P_{e}\Omega_{e}.
\end{align*}
 \vspace{-0.5cm}
\begin{theorem}\label{cost-difference-theorem-error-1}
Under the proposed distributed controllers \eqref{error-controller} and \eqref{observer-2} with $\Gamma_{i},i=1,2,\cdots,N$ chosen from
 the optimization  in \eqref{choose-gamma},
the corresponding cost function  is given by
\begin{align}\label{cost-function-distributed-control-error-homo}
&J^{\star}(s,\infty)
\notag\\
&=e^{T}(s)P_{e}e(s)+\sum_{k=s}^{\infty}\begin{bmatrix}
e(k)\\
\tilde{e}(k)
\end{bmatrix}^{T}
\begin{bmatrix}
0 & M_{e1}\\
M_{e1}^{T} & M_{e2}
\end{bmatrix}
\begin{bmatrix}
e(k)\\
\tilde{e}(k)
\end{bmatrix}
\end{align}
Moreover, the cost difference between the cost function \eqref{cost-function-distributed-control-error-homo} and the cost under the centralized optimal control
is given by
\begin{align}\label{cost-difference-error-1}
 \Delta J(s,\infty)&=J^{\star}(s,\infty)-J^{*}(s,\infty)
 \notag\\
 &=\sum_{k=s}^{\infty}\begin{bmatrix}
e(k)\\
\tilde{e}(k)
\end{bmatrix}^{T}
\begin{bmatrix}
0 & M_{e1}\\
M_{e1}^{T} & M_{e2}
\end{bmatrix}
\begin{bmatrix}
e(k)\\
\tilde{e}(k)
\end{bmatrix}.
\end{align}
In particular, the optimal cost function difference will approach to zero as $s$ is sufficiently large.
That is to say, the proposed  consensus controller  can achieve the optimal cost (asymptotically).
\end{theorem}

\begin{proof}
According to \eqref{algebra-riccati-equation-error} and \eqref{gloabl-error-closed-loop-system},
  we have
  \begin{align*}
  &e^{T}(k)P_{e}e(k)-e^{T}(k+1)P_{e}e(k+1)
  \notag\\
  &=e^{T}(k)(\mathcal{Q}+K_{e}^{T}RK_{e})e(k)
\notag\\
  &\quad -\tilde{e}^{T}(k)\Omega_{e}^{T}P_{e}(\tilde{A}+\bar{B}K_{e})e(k)
  \notag\\
  &\quad -e^{T}(k)(\tilde{A}+\bar{B}K_{e})^{T}P_{e}\Omega_{e}\tilde{e}(k)
  -\tilde{e}^{T}(k)\Omega_{e}^{T}P_{e}\Omega_{e} \tilde{e}(k)
  \end{align*}
  Based on the cost function \eqref{cost-function-error} under the centralized optimal control,
   by performing summation on $k$ from $s$ to $\infty$ and applying algebraic calculations yields
  \begin{align}
  &e^{T}(s)P_{e}e(s)-e^{T}(\infty)P_{e}e(\infty)
  \notag\\
  &=J(s,\infty)-\sum_{k=s}^{\infty}\begin{bmatrix}
e(k)\\
\tilde{e}(k)
\end{bmatrix}^{T}
\begin{bmatrix}
0 & M_{e1}\\
M_{e1}^{T} & M_{e2}
\end{bmatrix}
\begin{bmatrix}
e(k)\\
\tilde{e}(k)
\end{bmatrix}
  \end{align}
 It follows from Theorem \ref{main-result-2} that
 $\lim_{k\to \infty} e^{T}(k)P_{e}e(k)=0$.
 Therefore,
 the corresponding optimal cost function $J^{\star}(s,\infty)$ under the  proposed distributed observer-based controller \eqref{error-controller}
  is derived in \eqref{cost-function-distributed-control-error-homo}.
 Furthermore, in line with \eqref{cost-function-centralized-control-error-feedback},
 the optimal cost difference \eqref{cost-difference-error-1} holds.
 \\
According to Theorem \ref{main-result-2}, the closed-loop system \eqref{closed-loop-error-system-1} is stable,
then there exist two constants $a>0$ and $0<\gamma<1$ such that
\begin{align}
\left\|\begin{bmatrix}
e(k)\\
\tilde{e}(k)
\end{bmatrix}\right\|\leq  a\gamma^{k}\left\|\begin{bmatrix}
e(0)\\
\tilde{e}(0)
\end{bmatrix}\right\|.
\end{align}
Then, we have
\begin{align}\label{cost-difference-estimtaor}
 \Delta J_{e}(s,\infty)
 &=\sum_{k=s}^{\infty}\begin{bmatrix}
e(k)\\
\tilde{e}(k)
\end{bmatrix}^{T}
\begin{bmatrix}
0 & M_{e1}\\
M_{e1}^{T} & M_{e2}
\end{bmatrix}
\begin{bmatrix}
e(k)\\
\tilde{e}(k)
\end{bmatrix}
\notag\\
&\leq \sum_{k=s}^{\infty}
\left\|
\begin{bmatrix}
0 & M_{e1}\\
M_{e1}^{T} & M_{e2}
\end{bmatrix}
\right\|
\left\|\begin{bmatrix}
e(k)\\
\tilde{e}(k)
\end{bmatrix}\right\|^{2}
\notag\\
&\leq \left\|
\begin{bmatrix}
0 & M_{e1}\\
M_{e1}^{T} & M_{e2}
\end{bmatrix}
\right\|
\left\|\begin{bmatrix}
e(0)\\
\tilde{e}(0)
\end{bmatrix}\right\|^{2}\sum_{k=s}^{\infty}a^{2}\gamma^{2k}
\notag\\
&=\bar{a}\gamma^{2s}
\end{align}
with $\bar{a}= \left\|
\begin{bmatrix}
0 & M_{e1}\\
M_{e1}^{T} & M_{e2}
\end{bmatrix}
\right\|
\left\|\begin{bmatrix}
e(0)\\
\tilde{e}(0)
\end{bmatrix}\right\|^{2}\frac{a^{2}}{1-\gamma^{2}}$.
\\
Since $0<\gamma<1$,
for any given $\varepsilon>0$,
there exists a sufficiently large integer $M$ such that
$\gamma^{2M}<\frac{\varepsilon}{\bar{a}}$.
Based on \eqref{cost-difference-estimtaor},
it holds that
\begin{align}
\sum_{k=M}^{\infty}\begin{bmatrix}
e(k)\\
\tilde{e}(k)
\end{bmatrix}^{T}
\begin{bmatrix}
0 & M_{e1}\\
M_{e1}^{T} & M_{e2}
\end{bmatrix}
\begin{bmatrix}
e(k)\\
\tilde{e}(k)
\end{bmatrix}<\varepsilon.
\end{align}
In this case, when $M$ is large enough, the cost difference \eqref{cost-difference-error-1} satisfies
\begin{align*}
\Delta J(M,\infty)<\varepsilon.
\end{align*}
That is to say, the optimal cost difference \eqref{cost-difference-error-1} is equal to 0 as $s\to \infty$.
This proof is completed.
\end{proof}

\subsection{Comparison with traditional consensus algorithms}
Firstly, the consensus with the new controller has a faster convergence speed than the traditional consensus algorithms.

In fact,
from the closed-loop system \eqref{closed-loop-error-system-1},
  one has
\begin{align}\label{norm-estimator-error}
\left\|\begin{bmatrix}
e(k)\\
\tilde{e}(k)
\end{bmatrix}\right\|\leq \rho(\bar{A}_{ec})\left\|\begin{bmatrix}
e(k-1)\\
\tilde{e}(k-1)
\end{bmatrix}\right\|,
\end{align}
where $\rho(\bar{A}_{ec})$ is the spectra radius of  $\tilde{A}+\bar{B}K_{e}$ and $\tilde{A}_{ec}$.
In particularly,
 $\tilde{A}+\bar{B}K_{e}$ is the closed-loop system matrix obtained by the optimal feedback control \eqref{feedback-gain-matrix-Ke},
 that is,
 $e(k+1)=(\tilde{A}+\bar{B}K_{e})e(k)$ while $e(k+1)^{T}\tilde{Q}e(k+1)$  is minimized as in \eqref{cost-function-error},
so the modulus of the eigenvalues for $\tilde{A}+\bar{B}K_{e}$ is minimized in certain sense.
Besides, based on  the optimization in \eqref{choose-gamma} , we can appropriately select $\Upsilon_{i}$
such that the upper bound of the spectral radius $\rho( \tilde{A}_{ec})$ is as small as possible.
From these perspectives, $\rho(\bar{A}_{ec})$ is made more to be small. This is in comparison with the conventional consensus algorithms where
 the maximum eigenvalue of the matrix $\bar{A}_{ec}$ is not minimized and determined by the eigenvalues of the Laplacian matrix $\mathcal{L}$.
Therefore, it can be expected that the proposed approach can achieve a faster convergence than the conventional algorithms
 as demonstrated in the simulation examples in Section \ref{sec:example}.


Second, the cost difference $\Delta J(s,\infty)$ between the new distributed controller \eqref{error-controller} and the centralized optimal control \eqref{centralized-optimal-control-error}
is provided in Theorem \ref{cost-difference-theorem-error-1},
and it is equal to zero as $s\to\infty$.
 That is to say,  the corresponding cost function under the proposed distributed controllers \eqref{error-controller} is asymptotically optimal.

\subsection{Special case: consensus of multi-agent systems \eqref{homogeneous multi-agent system} via state feedback controller}


By stacking the state vectors, the multi-agent systems \eqref{homogeneous multi-agent system} can be rewritten as
\begin{align}\label{compact-system-homogeneous-1}
X(k+1)=\tilde{A}X(k)+\tilde{B}u(k)
\end{align}
where $X(k)=\begin{bmatrix}
x_{1}^{T}(k),\cdots,x_{N}^{T}(k)
\end{bmatrix}^{T}$ is the global state variable.
$\tilde{A}=I_{N}\otimes A$ and
$\tilde{B}=diag\{B_{1},B_{2},\cdots,B_{N}\}$.

 The cost function  \eqref{cost-function} is rewritten as
\begin{align}\label{cost-function-special-case}
J(s,\infty)
=\sum_{k=s}^{\infty}[X^{T}(k)\mathcal{Q}X(k)+u^{T}(k)Ru(k)]
\end{align}
where  $\mathcal{Q}=[\mathcal{Q}]_{ij}\geq 0$ with $\mathcal{Q}_{ii}=(N-1)Q$,
 $\mathcal{Q}_{ij}=-Q$ for $i\neq j$, and $R=diag\{R_{1},R_{2},\cdots,R_{N}\}>0$.


 \begin{lemma}\label{LQR-control-state}
Assume the state information $X(k)$ is available for all agents subject to \eqref{homogeneous multi-agent system},
the optimal controller with respect to the cost function \eqref{cost-function} is given by:
\begin{align}\label{centralized-optimal-control}
u^{*}(k)=KX(k),
\end{align}
where the feedback gain matrix $K=\mathcal{L}\otimes F$ is given as
\begin{align}\label{feedback-gain-matrix-K}
K=-(R+\tilde{B}^{T}P\tilde{B})^{-1}\tilde{B}^{T}P\tilde{A}
\end{align}
and $P$ is the solution of the following ARE
\begin{align}\label{algebra-riccati-equation}
P=\tilde{A}^{T}P\tilde{A}+\mathcal{Q}-\tilde{A}^{T}P\tilde{B}(R+ \tilde{B}^{T}P\tilde{B})^{-1}\tilde{B}^{T}P\tilde{A}.
\end{align}
The corresponding optimal cost function is
\begin{align}\label{cost-function-centralized-control}
J^{*}(s,\infty)=X^{T}(s)PX(s).
\end{align}
Moreover,  if $P$ is  the unique positive definite solution to \eqref{algebra-riccati-equation}, $\tilde{A}+\tilde{B}K$ is stable.
\end{lemma}



The system \eqref{compact-system-homogeneous-1} is rewritten as
\begin{align}\label{compact-system-homogeneous}
X(k+1)&=\tilde{A}X(k)+\sum_{i=1}^{N}\tilde{B}_{i}u_{i}(k)
\\
Y_{i}(k)&=C_{i}X(k)
\end{align}
where $\tilde{B}_{i}=\begin{bmatrix}
0 &  \cdots & B^{T} & 0 &\cdots & 0
\end{bmatrix}^{T}$,
  $Y_{i}(k)$ is measurement, $C_{i}$ is composed of $0$ and $I_{n}$, which is determined by the interaction among agents.

We design the distributed observer-based state feedback controllers:
\begin{align}\label{sate-feedback-controller}
u_{i}^{*}(k)&=K_{i}\hat{X}_{i}(k),\quad i=1,2\cdots,N
\end{align}
where the distributed observers $\hat{X}_{i}(k)$ are given by
\begin{align}\label{observer-state}
\hat{X}_{i}(k+1)&=\tilde{A}\hat{X}_{i}(k)+\tilde{B}_{1}K_{1}\hat{X}_{i}(k)+\cdots
\notag\\
&\quad+\tilde{B}_{i-1}K_{i-1}\hat{X}_{i}(k)+\tilde{B}_{i}u_{i}^{*}(k)
\notag\\
&\quad +\cdots+\tilde{B}_{N}K_{N}\hat{X}_{N}(k)
\notag\\
&\quad +L_{i}(Y_{i}(k)-C_{i}\hat{X}_{i}(k))
\end{align}
%
with $L_{i}$  the observer gains to be designed,
and $K_{i}=\begin{bmatrix}
0 & 0 &\cdots & I &\cdots & 0
\end{bmatrix}K$.


\begin{theorem}\label{main-result-1}
Let Assumption \ref{graph-assumption} hold. Consider the global system \eqref{compact-system-homogeneous} for the multi-agent systems \eqref{homogeneous multi-agent system},
and the control laws  \eqref{sate-feedback-controller},
if there exist observer gains $L_{i}$ such that the matrix
\begin{align}
\tilde{A}_{c}=\begin{bmatrix}
W_1 & -\tilde{B}_{2}K_{2} &  \cdots & -\tilde{B}_{N}K_{N}\\
-\tilde{B}_{1}K_{1} & W_2 &  \cdots & -\tilde{B}_{N}K_{N}\\
\vdots & \vdots & \ddots & \vdots \\
-\tilde{B}_{1}K_{1} & -\tilde{B}_{2}K_{2}& \cdots & W_N
\end{bmatrix}
\end{align}
is stable with $W_{i}=\tilde{A}+\tilde{B}K-\tilde{B}_{i}K_{i}-L_{i}C_{i}$.
Then the observers \eqref{observer-state} are stable, that is,
\begin{align}\label{observer-error-vector}
\lim_{k\to \infty}\|\hat{X}_{i}(k)-X(k)\|=0, i=1,\cdots,N.
\end{align}
Moreover, if $P$ is the positive definite solution of \eqref{algebra-riccati-equation}, under the feedback controller \eqref{sate-feedback-controller} ,
the multi-agent systems \eqref{homogeneous multi-agent system} can achieve consensus.
\end{theorem}
\begin{proof}
This proof is similar to that in Theorem \ref{main-result-2}. So we will
not repeat the details.
\end{proof}

Through analyzing Theorem \ref{main-result-1}, under the distributed state feedback controllers  \eqref{sate-feedback-controller} and \eqref{observer-state},
all agents can achieve consensus, and the state of each agent converges to zero, which is  also consistent with the single system's result \cite{Xu2023}.

Similar to Theorem \ref{cost-difference-theorem-error-1}, we will also discuss asymptotical optimal property of the
new distributed controllers \eqref{sate-feedback-controller}.
\begin{theorem}\label{cost-difference-theorem}
 Under the proposed distributed controller \eqref{sate-feedback-controller} and \eqref{observer-state} with $L_{i},i=1,2,\cdots,N$ chosen from
Theorem \ref{main-result-1},
the cost function is given by
\begin{align}\label{cost-function-distributed-control}
&J^{\star}(s,\infty)
\notag\\
&=X^{T}(s)PX(s)+\sum_{k=s}^{\infty}\begin{bmatrix}
X(k)\\
\tilde{X}(k)
\end{bmatrix}^{T}
\begin{bmatrix}
0 & M_{1}\\
M_{1}^{T} & M_{2}
\end{bmatrix}
\begin{bmatrix}
X(k)\\
\tilde{X}(k)
\end{bmatrix},
\end{align}
where
\begin{align*}
M_{1}&=(\tilde{A}+\tilde{B}K)^{T}P\Omega-
\begin{bmatrix}
K_{1}^{T}R_{1}K_{1} &  \cdots & K_{N}^{T}R_{N}K_{N}
\end{bmatrix},
\notag\\
M_{2}&=\begin{bmatrix}
K_{1}^{T}R_{1}K_{1} &  0 &\cdots & 0\\
0  & K_{2}^{T}R_{2}K_{2} &\cdots & 0\\
\vdots & \vdots& \ddots & \vdots\\
0 & 0 &\cdots & K_{N}^{T}R_{N}K_{N}
\end{bmatrix}+\Omega^{T}P\Omega.
\end{align*}
Moreover, the cost difference between the cost function \eqref{cost-function-distributed-control} and the optimal cost \eqref{cost-function-centralized-control}
is given by
\begin{align}\label{cost-difference}
 \Delta J(s,\infty)&=J^{\star}(s,\infty)-J^{*}(s,\infty)
 \notag\\
 &=\sum_{k=s}^{\infty}\begin{bmatrix}
X(k)\\
\tilde{X}(k)
\end{bmatrix}^{T}
\begin{bmatrix}
0 & M_{1}\\
M_{1}^{T} & M_{2}
\end{bmatrix}
\begin{bmatrix}
X(k)\\
\tilde{X}(k)
\end{bmatrix}.
\end{align}
Specially, when $s$ is sufficiently large, the cost difference will be equal to zero,
i.e., the  proposed  consensus controller \eqref{sate-feedback-controller} can achieve the optimal cost (asymptotically).
\end{theorem}

\begin{proof}
This proof is similar to that in Theorem \ref{cost-difference-theorem-error-1}.
So the details are omitted.
\end{proof}

\section{Numerical Simulation}\label{sec:example}

In this section, we  validate the  the proposed theoretical results through the following numerical examples.


\begin{example}
\begin{figure}[!htbp]
\centering
\includegraphics[width=1in]{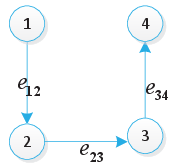}
\caption{Communication topology among four agents}
\label{fig:2}
\end{figure}
Consider the multi-agent system consisting of four homogeneous agents with
the system matrices taken from \cite{Chen2021},
\begin{align}
A=\begin{bmatrix}
1.1 & 0.3\\
0 & 0.8
\end{bmatrix},
B=\begin{bmatrix}
1 \\
0.5
\end{bmatrix}.
\end{align}
The interactions of agents are given in Fig.\ref{fig:2},
 in which each agent   receives neighbor error information, Then, we can determine $H_{i},i=1,2,3,4$ as
\begin{align*}
H_{1}&=\begin{bmatrix}
I_{2} & 0 & 0\\
\end{bmatrix},
H_{2}=\begin{bmatrix}
  0 & I_{2} & 0\\
\end{bmatrix},
\notag\\\
H_{3}&=\begin{bmatrix}
  0 & I_{2} & 0\\
\end{bmatrix},
H_{4}=0.
\end{align*}
We choose
\begin{align*}
Q=I_{2}, R_{1}=R_{2}=R_{3}=R_{4}=1.
\end{align*}
According to ARE  \eqref{algebra-riccati-equation-error} and the optimization in  \eqref{choose-gamma},
 the feedback gains $K_{ei}$ and the observer gain can be obtained, respectively.
Fig.\ref{fig:2-1} displays the evolution of each agent's state by the proposed consensus algorithm(M2),
it's shown that all agent's states reach the consensus value  after 12 steps.
The corresponding observer error vector $e_{1}(k)-\hat{e}_{1}(k)$ under the proposed controller \eqref{error-controller} converges to zero.
With the same initial conditions, Fig.\ref{fig:2-3} shows the state trajectories by the traditional state feedback method (M1) \cite{Olfati-Saber2004}.
One can see that the second state of each agent reach consensus after 25steps.
To further compare the convergence performance of different consensus algorithms,
the quantitative calculations based on the spectral radius $\rho(\tilde{A}_{ec})$
and the norm of the first agent's state at different instants of time are shown in Table \ref{Table-1}.
It can be observed that the proposed consensus algorithm reduces the maximum eigenvalue's value of $\tilde{A}_{c}$,
 and the norm of each agent's state $\|x_{i}(k_{0})\|$.
 Therefore,  the proposed distributed  observer-based consensus algorithm \eqref{error-controller} can ensure all agents reach consensus
 with a faster convergence speed.

\begin{figure}[!htbp]
\centering
\includegraphics[width=2.5in]{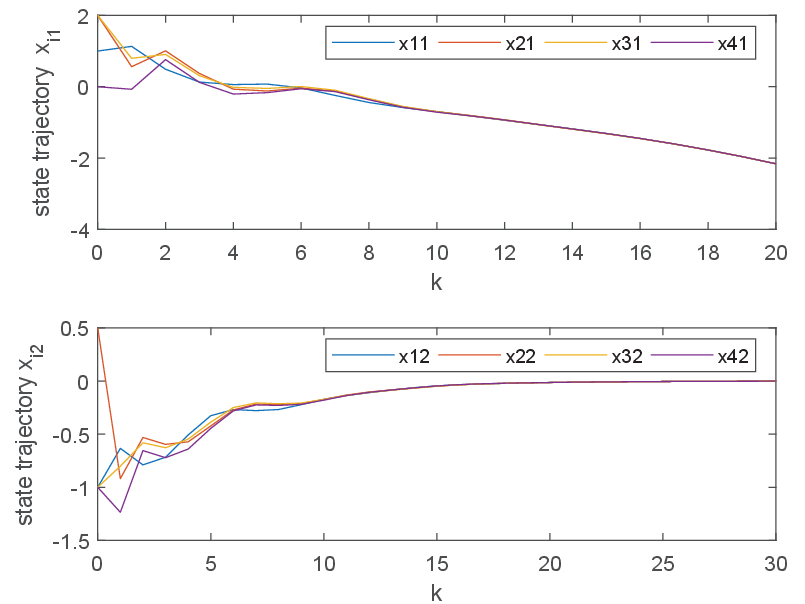}
\caption{The state trajectories for each agent $x_{i}(k), i=1,2,3,4$.}
\label{fig:2-1}
\end{figure}

\begin{figure}[!htbp]
\centering
\includegraphics[width=2.5in]{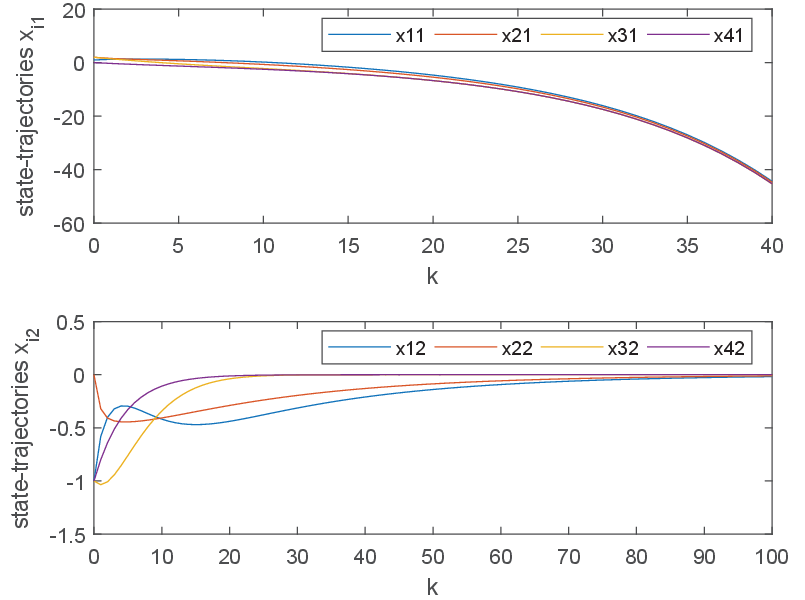}
\caption{The state trajectories $x_{i}(k)$ by the traditional consensus method.}
\label{fig:2-3}
\end{figure}
\begin{figure}[!htbp]
\centering
\includegraphics[width=2.5in]{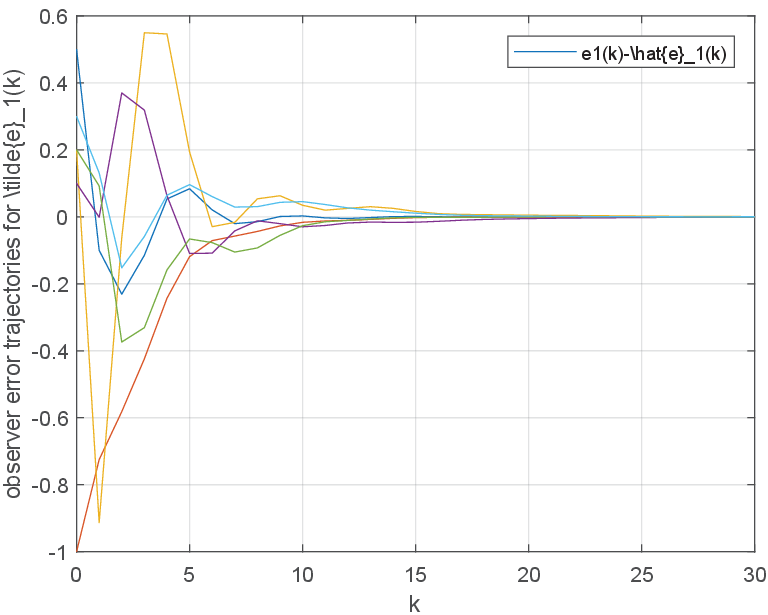}
\caption{Observer error trajectories $\tilde{e}_{1}(k)$.}
\label{fig:2-2}
\end{figure}

 \begin{table*}[t]
\centering
\caption{$\rho(\tilde{A}_{ec})$ and the norm of the first agent's state by using different algorithms}.
\label{table1}
\begin{tabular}{|c|c|c|c|c|c|c|c|c|c|c|c|c|}
\hline
\multirow{2}{*}{Method} &\multirow{2}{*}{$\rho(\tilde{A}_{ec})$} & \multicolumn{11}{c|}{$\|x_{1}\|_{2}$}\\
\cline{3-13}
  & & step1 & step2 & step4 & step6 & step8 & step10 & step12 & step14 & step16 & step18  & step20\\
\hline
M1 & 0.8777	& 1.4142 & 1.3670  & 1.3639  & 1.2424 & 0.9817  &  0.6218 &   0.4494  & 0.9513  &  1.7692 & 2.7805  &  3.9881  \\
\hline
M2 &0.7876	& 1.4142 & 1.2972 & 0.7293    & 0.3349  & 0.3702  & 0.6243 & 0.8271   & 1.0638  & 1.3096  & 1.6056   & 1.9587  \\
\hline
\end{tabular}
\label{Table-1}
\end{table*}
\end{example}

\begin{example}
 \begin{figure}[!htbp]
\centering
\includegraphics[width=1in]{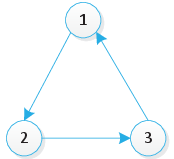}
\caption{Communication topology among three agents}
\label{fig:1}
\end{figure}

We first consider a multi-agent system that consists of three agents over a directed graph described in Fig.\ref{fig:1}.
The system dynamic parameters are set as
 \begin{align*}
 A&=1,
 B_1=\begin{bmatrix}
 1.5 & 0.5
 \end{bmatrix},
 B_{2}=\begin{bmatrix}
 0.8 & 1
 \end{bmatrix},
 B_{3}=\begin{bmatrix}
 1& -0.2
 \end{bmatrix},
 \notag\\
 C_{1}&=\begin{bmatrix}
1 & 0 & 0\\
0 & 0 & 1
\end{bmatrix},
C_{2}=\begin{bmatrix}
1 & 0 & 0\\
0 & 1 & 0
\end{bmatrix},
C_{3}=\begin{bmatrix}
0 & 1 & 0\\
0 & 0 & 1
\end{bmatrix},
\notag\\
Q&=1, R_1=R_2=R_3=1.
 \end{align*}
By solving ARE \eqref{algebra-riccati-equation} and according to Theorem \ref{main-result-1}, the feedback gains in \eqref{feedback-gain-matrix-K} are obtained
\begin{align*}
K_1&=\begin{bmatrix}
   -0.3935  & -0.0000  &    -0.0000 \\
   -0.1312 &  -0.0000 &   -0.0000
    \end{bmatrix},
\notag\\
K_2&=\begin{bmatrix}
   -0.0000 &  -0.2849 &   0.0000  \\
   -0.0000 &  -0.3561 &   0.0000
    \end{bmatrix},
    \notag\\
K_{3}&=\begin{bmatrix}
   -0.0000 &   0.0000 &  -0.4871 \\
    0.0000 &  -0.0000 &   0.0974
    \end{bmatrix}.
\end{align*}
 We obtain observer gain matrices from the optimization in \eqref{choose-gamma} as
 \begin{align*}
 L_{1}&=\begin{bmatrix}
   1.0000  &  -0.0000 \\
   -0.0000  &  0.0000 \\
   -0.0000   & 0.4934
 \end{bmatrix},
 L_{2}=\begin{bmatrix}
  0.3441  & -0.0000 \\
    0.0000 &   1.0000 \\
   -0.0000 &   0.0000
 \end{bmatrix},
 \notag\\
 L_{3}&=\begin{bmatrix}
  -0.0000 &  -0.0000 \\
    0.4160  &  0.0000 \\
   -0.0000  &  1.0000
 \end{bmatrix}.
 \end{align*}
The observer error systems are stable as shown in Fig.\ref{fig:1-1}.
Meanwhile,
Fig.\ref{fig:1-2} shows the evolution of the multi-agent systems \eqref{homogeneous multi-agent system} under the distributed state feedback controllers
\eqref{sate-feedback-controller},
where one can see that all agents' states converge to zero within ten steps, which indicates consensus can be achieved rapidly.

\begin{figure}[!htbp]
\centering
\includegraphics[width=2.5in]{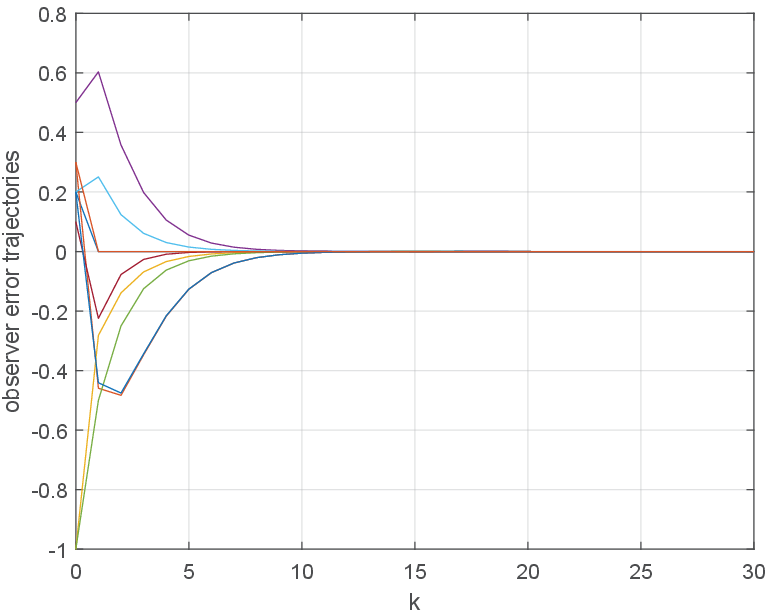}
\caption{Observer error trajectories $\tilde{X}_{i}(k)$ in Theorem \ref{main-result-1}.}
\label{fig:1-1}
\end{figure}

\begin{figure}[!htbp]
\centering
\includegraphics[width=2.5in]{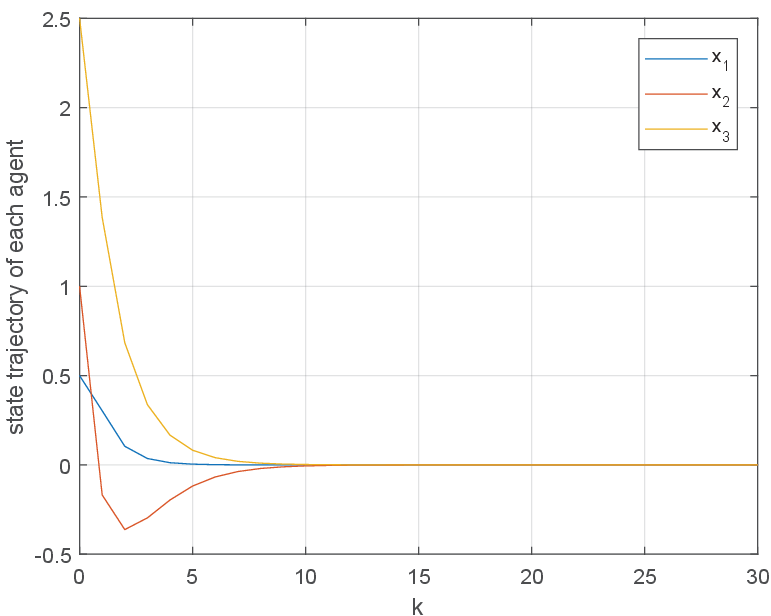}
\caption{The state trajectory of each agent $x_{i}(k),i=1,2,3$.}
\label{fig:1-2}
\end{figure}

\end{example}

\section{Conclusions}\label{sec:conclusion}
In this paper, we have studied the consensus problem for discrete-time linear multi-agent systems by LQ optimal control theory.
Different from the existing consensus algorithms,
we designed a novel distributed controller based on observers involving agent's historical state information by solving Riccati equations.
It's shown that the corresponding global cost function under the proposed controller is asymptotically optimal.
The new consensus algorithm does not require to compute the eigenvalues of the communication topology,
and can achieve a much faster consensus speed than the traditional consensus methods.
Finally,
simulation examples and comparisons with the existing consensus algorithm were provided to demonstrate the feasibility and effectiveness of the proposed control algorithms.



\ifCLASSOPTIONcaptionsoff
  \newpage
\fi



%
\bibliographystyle{IEEEtran}
\bibliography{mybib}

\begin{thebibliography}{10}
\providecommand{\url}[1]{#1}
\csname url@samestyle\endcsname
\providecommand{\newblock}{\relax}
\providecommand{\bibinfo}[2]{#2}
\providecommand{\BIBentrySTDinterwordspacing}{\spaceskip=0pt\relax}
\providecommand{\BIBentryALTinterwordstretchfactor}{4}
\providecommand{\BIBentryALTinterwordspacing}{\spaceskip=\fontdimen2\font plus
\BIBentryALTinterwordstretchfactor\fontdimen3\font minus
  \fontdimen4\font\relax}
\providecommand{\BIBforeignlanguage}[2]{{%
\expandafter\ifx\csname l@#1\endcsname\relax
\typeout{** WARNING: IEEEtran.bst: No hyphenation pattern has been}%
\typeout{** loaded for the language `#1'. Using the pattern for}%
\typeout{** the default language instead.}%
\else
\language=\csname l@#1\endcsname
\fi
#2}}
\providecommand{\BIBdecl}{\relax}
\BIBdecl

\bibitem{Ren2007a}
E.~M.~A. Wei~Ren, Randal W.~Beard, ``Information consensus in multivehicle
  cooperative control,'' \emph{{IEEE} Control Systems}, vol.~27, no.~2, pp.
  71--82, apr 2007.

\bibitem{Olfati-Saber2007}
R.~Olfati-Saber, J.~A. Fax, and R.~M. Murray, ``Consensus and cooperation in
  networked multi-agent systems,'' \emph{Proceedings of the {IEEE}}, vol.~95,
  no.~1, pp. 215--233, jan 2007.

\bibitem{Yang2022}
Y.~Yang, Y.~Xiao, and T.~Li, ``Attacks on formation control for multiagent
  systems,'' \emph{{IEEE} Transactions on Cybernetics}, vol.~52, no.~12, pp.
  12\,805--12\,817, dec 2022.

\bibitem{Li2014}
Z.~Li and Z.~Duan, ``Distributed consensus protocol design for general linear
  multi-agent systems: a consensus region approach,'' \emph{{IET} Control
  Theory {\&} Applications}, vol.~8, no.~18, pp. 2145--2161, dec 2014.

\bibitem{Olfati-Saber2004}
R.~Olfati-Saber and R.~Murray, ``Consensus problems in networks of agents with
  switching topology and time-delays,'' \emph{{IEEE} Transactions on Automatic
  Control}, vol.~49, no.~9, pp. 1520--1533, sep 2004.

\bibitem{Ren2005}
W.~Ren and R.~Beard, ``Consensus seeking in multiagent systems under
  dynamically changing interaction topologies,'' \emph{{IEEE} Transactions on
  Automatic Control}, vol.~50, no.~5, pp. 655--661, may 2005.

\bibitem{Ren2007}
W.~Ren and E.~Atkins, ``Distributed multi-vehicle coordinated controlvia local
  information exchange,'' \emph{International Journal of Robust and Nonlinear
  Control}, vol.~17, no. 10-11, pp. 1002--1033, 2007.

\bibitem{You2011}
K.~You and L.~Xie, ``Network topology and communication data rate for
  consensusability of discrete-time multi-agent systems,'' \emph{{IEEE}
  Transactions on Automatic Control}, vol.~56, no.~10, pp. 2262--2275, oct
  2011.

\bibitem{Feng2022}
T.~Feng, J.~Zhang, Y.~Tong, and H.~Zhang, ``Consensusability and global
  optimality of discrete-time linear multiagent systems,'' \emph{{IEEE}
  Transactions on Cybernetics}, vol.~52, no.~8, pp. 8227--8238, aug 2022.

\bibitem{Xu2021}
J.~Xu, Z.~Zhang, and W.~Wang, ``Mean-square consentability of multiagent
  systems with nonidential channel fading,'' \emph{{IEEE} Transactions on
  Automatic Control}, vol.~66, no.~4, pp. 1887--1894, apr 2021.

\bibitem{Xu2019}
J.~Xu, H.~Zhang, and L.~Xie, ``Consensusability of multiagent systems with
  delay and packet dropout under predictor-like protocols,'' \emph{{IEEE}
  Transactions on Automatic Control}, vol.~64, no.~8, pp. 3506--3513, aug 2019.

\bibitem{Liu2011}
J.~Liu, X.~Liu, W.~Xie, and H.~Zhang, ``Stochastic consensus seeking with
  communication delays,'' \emph{Automatica}, vol.~47, no.~12, pp. 2689--2696,
  dec 2011.

\bibitem{Cao2010}
Y.~Cao and W.~Ren, ``Optimal linear-consensus algorithms: An {LQR}
  perspective,'' \emph{{IEEE} Transactions on Systems, Man, and Cybernetics,
  Part B (Cybernetics)}, vol.~40, no.~3, pp. 819--830, jun 2010.

\bibitem{Chen2020}
F.~Chen and J.~Chen, ``Minimum-energy distributed consensus control of
  multiagent systems: A network approximation approach,'' \emph{{IEEE}
  Transactions on Automatic Control}, vol.~65, no.~3, pp. 1144--1159, mar 2020.

\bibitem{Su2012}
Y.~Su and J.~Huang, ``Cooperative output regulation of linear multi-agent
  systems,'' \emph{{IEEE} Transactions on Automatic Control}, vol.~57, no.~4,
  pp. 1062--1066, apr 2012.

\bibitem{Huang2017}
J.~Huang, ``The cooperative output regulation problem of discrete-time linear
  multi-agent systems by the adaptive distributed observer,'' \emph{{IEEE}
  Transactions on Automatic Control}, vol.~62, no.~4, pp. 1979--1984, apr 2017.

\bibitem{Lu2017}
M.~Lu and L.~Liu, ``Cooperative output regulation of linear multi-agent systems
  by a novel distributed dynamic compensator,'' \emph{{IEEE} Transactions on
  Automatic Control}, vol.~62, no.~12, pp. 6481--6488, dec 2017.

\bibitem{Wieland2011}
P.~Wieland, R.~Sepulchre, and F.~Allgöwer, ``An internal model principle is
  necessary and sufficient for linear output synchronization,''
  \emph{Automatica}, vol.~47, pp. 1068--1074, 2011.

\bibitem{Li2015}
S.~Li, G.~Feng, X.~Luo, and X.~Guan, ``Output consensus of heterogeneous linear
  discrete-time multiagent systems with structural uncertainties,''
  \emph{{IEEE} Transactions on Cybernetics}, vol.~45, no.~12, pp. 2868--2879,
  dec 2015.

\bibitem{Bi2022}
C.~Bi, X.~Xu, L.~Liu, and G.~Feng, ``Robust cooperative output regulation of
  heterogeneous uncertain linear multiagent systems with unbounded distributed
  transmission delays,'' \emph{{IEEE} Transactions on Automatic Control},
  vol.~67, no.~3, pp. 1371--1383, 2022.

\bibitem{Meng2017}
M.~Meng, L.~Liu, and G.~Feng, ``Output consensus for heterogeneous multiagent
  systems with markovian switching network topologies,'' \emph{International
  Journal of Robust and Nonlinear Control}, vol.~28, no.~3, pp. 1049--1061,
  Feb. 2017.

\bibitem{Liu2018}
T.~Liu and J.~Huang, ``Leader-following attitude consensus of multiple rigid
  body systems subject to jointly connected switching networks,''
  \emph{Automatica}, vol.~92, pp. 63--71, jun 2018.

\bibitem{Yaghmaie2017}
F.~A. Yaghmaie, R.~Su, F.~L. Lewis, and S.~Olaru, ``Bipartite and cooperative
  output synchronizations of linear heterogeneous agents: A unified
  framework,'' \emph{Automatica}, vol.~80, pp. 172--176, jun 2017.

\bibitem{Scardovi2009}
L.~Scardovi and R.~Sepulchre, ``Synchronization in networks of identical linear
  systems,'' \emph{Automatica}, vol.~45, no.~11, pp. 2557--2562, Nov. 2009.

\bibitem{Anderson1971}
J.~M. B.D.O.~Anderson, \emph{Linear optimal control}.\hskip 1em plus 0.5em
  minus 0.4em\relax Prentice Hall, 1971.

\bibitem{Xu2023}
\BIBentryALTinterwordspacing
J.~Xu and H.~Zhang, ``Decentralized control of linear systems with private
  input and measurement information,'' 2023. [Online]. Available:
  \url{https://doi.org/10.48550/arXiv.2305.14921}
\BIBentrySTDinterwordspacing

\bibitem{Chen2021}
G.~Chen, Y.~Kang, C.~Zhang, and S.~Chen, ``Consensus of discrete-time
  multi-agent systems over packet dropouts channels,'' \emph{Journal of the
  Franklin Institute}, vol. 358, no.~13, pp. 6684--6704, sep 2021.

\end{thebibliography}

\end{document}